\documentclass[10pt,leqno]{article}

\usepackage{amsmath,amssymb,amsthm,mathrsfs,dsfont}

\usepackage[margin=3cm]{geometry} 

\usepackage{titlesec,hyperref}

\usepackage{color}

\usepackage{fancyhdr}
\titleformat{\subsection}{\it}{\thesubsection.\enspace}{1pt}{}

\newtheorem{theo}{Theorem}[section]
\newtheorem{lemm}[theo]{Lemma}
\newtheorem{defi}[theo]{Definition}
\newtheorem{coro}[theo]{Corollary}

\numberwithin{equation}{section}

\allowdisplaybreaks 

\begin{document}
\title{Uniqueness of conservative solutions to the the modified Camassa-Holm equation via Characteristics
	\hspace{-4mm}
}

\author{Zhen $\mbox{He}^1$ \footnote{Email: hezh56@mail2.sysu.edu.cn},\quad
	Zhaoyang $\mbox{Yin}^{1,2}$\footnote{E-mail: mcsyzy@mail.sysu.edu.cn}\\
$^1\mbox{Department}$ of Mathematics,
Sun Yat-sen University, Guangzhou 510275, China\\
$^2\mbox{School}$ of Science,\\ Shenzhen Campus of Sun Yat-sen University, Shenzhen 518107, China}

\date{}
\maketitle
\hrule

\begin{abstract}
In this paper,for a given conservative solution, we introduce
a set of auxiliary variables tailored to this particular solution, and  prove that these
variables satisfy a particular semilinear system having unique solutions. In turn, we get
the uniqueness of the conservative solution in the original variables.
	
\vspace*{5pt}
	
	\vspace*{5pt}
\end{abstract}

\vspace*{10pt}

\tableofcontents

	\section{Introduction}
	
	  ~~In this paper,  we consider the Cauchy problem of the following modified Camassa-Holm (MOCH) equation \cite{2021}
	\begin{equation}\label{eq0}
		\left\{\begin{array}{l}
			\gamma_t=\lambda(v_x-\gamma-\frac{1}{\lambda}v\gamma)_x
			,  \quad t>0,\ x\in\mathbb{R},  \\
			v_{xx}=\gamma_x+\frac {\gamma^2} {2\lambda} ,  \quad t\geq0,\ x\in\mathbb{R},  \\
			\gamma(0,x)=\gamma_0(x),  \quad x\in\mathbb{R},
		\end{array}\right.
	\end{equation}
	which was called by Gorka and Reyes\cite{GR2010}. Let $G={\partial_x}^2-1,n=Gv$
	
	The equation \eqref{eq0} can be rewritten as
	\begin{equation}\label{eq1}
		\left\{\begin{array}{l}
			\gamma_t+G^{-1}n\gamma_x=\frac {\gamma^2}{2} +\lambda G^{-1}n-
			\gamma G^{-1}n_x,  \quad t>0,\ x\in\mathbb{R},  \\
			n=\gamma_x+\frac {\gamma^2} {2\lambda} ,  \quad t\geq0,\ x\in\mathbb{R},  \\
			\gamma(0,x)=\gamma_0(x),  \quad x\in\mathbb{R}.
		\end{array}\right.
	\end{equation}

	  The equation $\eqref{eq0}$ was first studied through the geometric approach in\cite{Chern} \cite{Reyes}. Conservation laws and the existence and uniqueness of weak solutions to the
	modified Camassa–Holm equation were presented in \cite{GR2010}. We observe that if we solve
	\eqref{eq1}, then $m$ will formally satisfy the following physical form of the Camassa–Holm

	\begin{align}
		n_t=-2vn_x-nv_x+\lambda v_x
	\end{align}

	  If $\lambda=0$, it is known as the well-known Camassa–Holm (CH) equation. 
	As far as we know, the CH equation has many properties, such as: integrability  \cite{Constantin2001,Constantin2006,Constantin1999},  Hamiltonian structure, infinitely many conservation laws \cite{Constantin2001,Fuchssteiner1981}. The local well-posedness  of  CH equation  has been proved in  Sobolev spaces $H^s, s>\frac 3 2$ and in Besov space $B^{s}_{p,r}$ with $ s>\max\{\frac{3}{2},1+\frac 1 p \}$ or $s=1+\frac 1 p,p\in [1,2],r=1$ in \cite{Constantin1998,Constantin1998w-l,Danchin2001inte,Danchin2003wp,Liu2019ill,Himonas2012CN,Li2016nwpC}. In \cite{Constantin2000gw,Xin2000ws}, the authors showed  the existence and uniqueness of global weak solutions for the CH equation. In addition, Bressan and Constantin studied  the existence of the global conservative solutions \cite{book}  and global disspative solutions \cite{Bressan2007gd}  in $H^1(\mathbb{R}).$  Later, Bressan and his collaborators used the similar technique to a variational wave equation\cite{WE1} and obtain the uniqueness \cite{WE2}. Utilizing the good structure of the semilinear ODE system proposed in \cite{WE1}, Bressan and Chen proved that, For an open dense
	set of $C^3$ initial data, we prove that the solution is piecewise smooth in the t-x plane, while the gradient $u_x$ can blow up along finitely many characteristic curves\cite{WE3}. Later, Li and Zhang studied similar property in \cite{LZ}
	
	  Luo, Qiao and Yin studied the locally well-posedness in $B^s_{p,r},s\textgreater \max\{\frac{1}{2},\frac{1}{p}\}$ or $s=\frac{1}{p},1\leq p\leq 2,r=1$, blow up phenomena, global existence for periodic MOCH and global conservative solution\cite{Luo1,Luo2,Luo3}.
	
	  The remainder of the paper is organized as follows. In Section 2 we review basic definitions
	  and state our main uniqueness result Theorem \ref{main} . Section 3 establishes the key technical tool
	  (Lemma \ref{lemm2}), determining a unique characteristic curve through each initial point. In Section 4
	  we conclude the proof of the main theorem. We can easily deduce that $\gamma=(\partial_x-1)m$ is a unique solution to \eqref{eq0}.

	\section{\textbf{Basic definitons and results}}
	Firstly we think about the initial problem
		\begin{equation}\label{F1}
		\left\{\begin{array}{l}
			x_t=u(t,x) \\
			
			x(0)=\bar{x}.
		\end{array}\right.
	\end{equation} 
		Now, let us consider a transformation~$m=(\partial_x+1)G^{-1}\gamma=(\partial_x-1)^{-1}\gamma$. Then we have  $\gamma=(\partial_x-1)m$, and therefore, equation \eqref{eq1} is changed to 
	\begin{equation}\label{ceg1}
	 m_t+um_x =F,
	\end{equation}
	where
	\begin{align}
			&	F=mu-P_4-P_{4x}+\frac{1}{2}(-m^2+P_3+P_{3x})+\lambda P_{1x}+\frac{1}{2}(P_5+P_{5x}-P_2)&
			\notag\\
	&	P_1=G^{-1}m=p\ast m,~~~~~P_2=p\ast m^2,~~~~~P_3=p\ast m_x^2,&
		\notag\\
	&	u=m-P_{1x}+P_1+\frac {1} {2\lambda}(P_3+P_2-P_{2x}),~~~~~P_4=p\ast H,&
		\notag \\
	&	P_5=p\ast P_3,~~~~~~H=um_x-um,~~~~~~p=\frac{1}{2}e^{-|x|}.&
	\notag
	\end{align}
	 
	For the initial data, we have 
	\begin{equation}\label{3ini}
		m(0,x)=\bar{m}(x)=(\partial_x-1)^{-1}\bar{\gamma}(x).
	\end{equation}
	For smooth solutions, differentiating \eqref{ceg1} w.r.t. x one obtains
	\begin{equation}\label{ceg2}
		m_{xt}+u_xm_x+um_{xx}=F_x,
	\end{equation}
	multiplying \eqref{ceg2} with $m_x$, we obtain
	\begin{equation}\label{balance1}
		\frac{d}{2dt}m_x^2+\frac{1}{2}(um_x^2)_x=E,
	\end{equation}
 where
 \begin{align}
 	&	E=(mu)_xm_x-P_{4x}m_x-Hm_x-P_4m_x+\frac{1}{2}(-2mm_x^2+P_{3x}m_x+m_x^3+P_{3}m_x)&
 		\notag\\
 	&~~~~~~	+\lambda (m+P_{1})m_x+\frac{1}{2\lambda}(P_{5x}+P_5+P_3-P_{2}+m^2)m_x-\frac{1}{2}u_xm_x^2&
 	\notag\\
 		&~~~	=(mu)_xm_x-P_{4x}m_x-Hm_x-P_4m_x+\frac{1}{2}(P_{3x}m_x+P_{3}m_x)+\lambda (m+P_{1})m_x&
 	\notag\\
 	&~~~~~~	+\frac{1}{2\lambda}(P_{5x}+P_5+P_3-P_{2}+m^2)m_x&
 	\notag\\
 	&~~~~~~-\frac{1}{2}(mm_x^2-P_1m_x^2+m_x^2P_{1x}-\frac{1}{2\lambda}(P_3+P_2-P_{2x})m_x^2).&
 \end{align}
We denote $m_x^2$ by $w$, then 
\begin{equation}\label{ceg3}
\frac{d}{dt}w+(uw)_x=2E.
\end{equation} 
Because of \eqref{ceg3}, the characteristic curve $t\rightarrow x(t)$ satisies the additional eqaution
\begin{equation}\label{ceg4}
	\frac{d}{dt}\int_{-\infty}^{x(t)}m_x^2(t,x)dx=\int_{-\infty}^{x(t)}2E(t,x)dx.
\end{equation}
And we introduce new coordinates (t,$\beta$), where $\beta$ is implictly defined as 
\begin{align}
	x(t,\beta)+\int_{-\infty}^{x(t,\beta)}m_x^2(t,\xi)d\xi=\beta.
\end{align}
At a time t where the measure $\mu_{(t)}$ is not absolutely continuous w.r.t Lebesgue measure, we can define x(t,$\beta$) to be the unique point x such that
\begin{align}\label{def1}
	x(t,\beta)+\mu_{(t)}(]-\infty,x[)\leq \beta \leq x(t,\beta)+\mu_{(t)}(]-\infty,x])
\end{align}
 Then we shall prove Lipshcitz continuous of t and u as functions of the variables t,$\beta$.
 Before providing our main results in this paper, let us first introduce some definitions of the globally conservative solutions for \eqref{eq0} and \eqref{ceg1}
 \begin{defi}\label{3defi0}
 	Let~ $\bar{\gamma}\in L^2(\mathbb{R})$. We say that~ $\gamma(t,x)$  is locally conservative solution to the Cauchy problem \eqref{eq0} if $\gamma$ satisies the following equality:
 	\begin{align*}
 		&\int_{\mathbb{R}^+}\int_{\mathbb R}\big(\gamma\psi_{t}+(\gamma\partial_xG^{-1}\gamma+\frac 1 {2\lambda}\gamma G^{-1}\gamma^2)\psi_x\big)(t,x)dx dt \\
 		&=-\int_{\mathbb{R}^+}\int_{\mathbb R}\big((\frac {\gamma^2} 2 +\lambda\partial_xG^{-1}\gamma+\frac 1 2  G^{-1}\gamma^2)\psi\big)(t,x)dxdt-\int_{\mathbb{R}}\bar{\gamma}(x)\psi(0,x)dx,
 	\end{align*}
 	for every test function $\phi\in C_c^\infty(\mathbb{R}^2)$
 \end{defi}
 \begin{defi}
 	A solution m=m(t,,x) is conservative if $w= m_x^2$ provides a distributional solution to the balance law \eqref{balance1}, namely
 	\begin{align}\label{bal}
 		\int_{0}^{\infty}\int [m_x^2\phi_t+um_x^2\phi_x+2E\phi]dxdt+\int m_{0,x}^2\phi(0,x)dx =0
 	\end{align}
 	for every test function $\phi\in C_c^1(\mathbb{R}^2)$
 \end{defi}
  Our main theorem is stated as follows.
 \begin{theo}\label{main}
 	For any initial data $\bar{m}\in H^1(\mathbb{R})$, the Cauchy problem \eqref{ceg1} has a unique conservative solution
 \end{theo}
	\begin{coro}
	For any initial data $\gamma_0\in L^2(\mathbb{R})$, the Cauchy problem \eqref{eq0} has a unique conservative solution
\end{coro}
\section{Preliminary lemmas}
\begin{lemm}\label{lemm1}
	Let m=m(t,x) be a conservative solution of equation \eqref{ceg1}. Then, for every $t\geq 0$, the maps $\beta\rightarrow x(t,\beta)$ and $\beta\rightarrow m(t,\beta)\triangleq m(t,x(t,\beta))$ implicitly defined by \eqref{def1} are Lipschitz continuous with constant 1. The map t $\rightarrow x(t,\beta)$ is also Lipschitz continuous with constant 1. The map $t\rightarrow x(t,\beta)$ is also Lipschitz continuous with a constant deppending only on $\|m_0\|_{H^1}$
\end{lemm}
\begin{proof}
	Fix any time t$\geq 0$. Then the map 
	\begin{align}
		x\rightarrow \beta(t,x)\triangleq x+\int_{-\infty}^{x} m_x^2(t,y)dy
	\end{align}
is right continuous and strictly increasing. Hence it has a well defined, continuous, nondecreasing inverse $\beta\rightarrow x(t,\beta)$. If $\beta_1<\beta_2$, by \eqref{def1} we obtain
\begin{align}\label{def2}
	x(t,\beta_2)-x(t,\beta_1)+\mu_{(t)}(]-\infty,x(\beta_2)[)-\mu_{(t)}(]-\infty,x(\beta_1)])\leq \beta_2-\beta_1,
\end{align} 
which implies 
\begin{align}
	x(t,\beta_2)-x(t,\beta_1)\leq \beta_2-\beta_1
\end{align}
showing that the map $\beta \rightarrow x(t,\beta)$ is a contraction.

According to \eqref{def2}
\begin{align}\label{MMM}
	|m(t,x(t,\beta_2))-m(t,x(t,\beta_1))|&\leq\int_{x(t,\beta_1)}^{x(t,\beta_2)}|m_x|dx\leq \int_{x(t,\beta_1)}^{x(t,\beta_2)}\frac{1}{2}|1+m_x^2|dx
	\notag\\
		&=\frac{1}{2}(x(t,\beta_2)-x(t,\beta_1))+\int_{x(t,\beta_1)}^{x(t,\beta_2)}\frac{1}{2}|m_x^2|dx
		\notag\\
		&\leq \frac{1}{2}(x(t,\beta_2)-x(t,\beta_1))+\frac{1}{2}\mu_{(t)}(]x(t,\beta_1),x(t,\beta_2)[)
		\notag\\
		&\leq \frac{1}{2}(\beta_2-\beta_1).
\end{align}
Next we prove the Lipschitz continuity of the map $t\rightarrow x(t,\beta)$. By Sobolev's inequality, we have
\begin{align}
	\|m\|_{L^\infty}&\leq \|m\|_{H^1}.
\end{align}
Using Young's inequality, we get
\begin{align}
	\|P_{1}\|_{L^\infty} &\leq \|e^{-|\cdot|}\|_{L^1}\|m\|_{L^\infty}\leq C\|m\|_{H^1},~~~
	\|P_{1x}\|_{L^\infty} \leq \|sgn(\cdot)e^{-|\cdot|}\|_{L^1}\|m\|_{L^\infty}\leq C\|m\|_{H^1},\label{p1}
	\\
	\|P_{2}\|_{L^\infty} &\leq \|e^{-|\cdot|}\|_{L^\infty}\|m^2\|_{L^1}\leq C\|m\|_{H^1},~~~
	\|P_{2x}\|_{L^\infty} \leq \|sgn(\cdot)e^{-|\cdot|}\|_{L^\infty}\|m^2\|_{L^1}\leq C\|m\|_{H^1}, \label{p2}
	\\
	\|P_{3}\|_{L^\infty} &\leq \|e^{-|\cdot|}\|_{L^\infty}\|m_x^2\|_{L^1}\leq C\|m\|_{H^1},~~~
	\|P_{3x}\|_{L^\infty} \leq \|sgn(\cdot)e^{-|\cdot|}\|_{L^\infty}\|m_x^2\|_{L^1}\leq C\|m\|_{H^1}, \label{p3}
\end{align}
From \eqref{p1}-\eqref{p3}, we can deduce that
\begin{align}
	\|u\|_{L^\infty}\leq C_{\lambda}\|m\|_{H^1} \triangleq C_\infty.
\end{align}
Similarly 
\begin{align}
	\|P_{5}\|_{L^\infty} &\leq \|e^{-|\cdot|}\|_{L^1}\|P_3\|_{L^\infty}\leq C\|m\|_{H^1},~~~
	\|P_{5x}\|_{L^\infty} \leq \|sgn(\cdot)e^{-|\cdot|}\|_{L^1}\|P_3\|_{L^\infty}\leq C\|m\|_{H^1}. \label{p5}
\end{align}
To estimate $\|P_4\|_{L^\infty}$, we can see
\begin{align}
	\|P_4\|_{L^\infty}=\|p\ast H\|_{L^\infty} \leq \|e^{|\cdot|}\|_{L^\infty}\|um_x\|_{L^1}+\|e^{|\cdot|}\|_{L^1}\|um\|_{L^\infty}.
\end{align}
Then the main diffculty that confronts us is the estimation of the term $\|um_x\|_{L^1}$. 

Firstly, by Young's inequality
\begin{align}
	\|m_xP_3\|_{L^1} &\leq C\|(1+\frac{m_x^2}{2})G^{-1}m_x^2\|_{L^1}\leq C(\|G^{-1}m_x^2\|_{L^1}+ \|\frac{m_x^2}{2}G^{-1}m_x^2\|_{L^1})
	\notag\\
	&\leq C\|e^{-|\cdot|}\|_{L^\infty}\|m_x^2\|_{L^1}+\|\frac{m_x^2}{2}\|_{L^1}\|G^{-1}m_x^2\|_{L^\infty}
	\notag\\
	&\leq C(\|m_x^2\|_{L^1}+\|m_x^2\|_{L^1}\|e^{|\cdot|}\|_{L^\infty}\|m_x^2\|_{L^1})
	\notag\\
	&\leq C(\|m\|_{H^1}+\|m\|_{H^1}^2).
\end{align}
Similarly we have 
\begin{align}
    \|m_xP_{2}\|_{L^1} ,\|m_xP_{2x}\|_{L^1},\|m_xP_1\|_{L^1},\|m_xP_{1x}\|_{L^1} \leq C\|m\|_{H^1}^2.\label{mx}
\end{align}
Then one obtains
\begin{align}
	\|um_x\|_{L^1}\leq C(\|m\|_{H^1}+\|m\|_{H^1}^2).
\end{align}
Using the same mothed, we can deduce that
\begin{align}
	\|P_4\|_{L^\infty}\leq C(\|m\|_{H^1}+\|m\|_{H^1}^2),~~~	\|P_{4x}\|_{L^\infty}\leq C(\|m\|_{H^1}+\|m\|_{H^1}^2).\label{p4}
\end{align}
Next, by Young's inequality, we have 
\begin{align}
	\|P_1\|_{L^1},~\|P_2\|_{L^1},~\|P_3\|_{L^1},~\|P_{3x}\|_{L^1},~\|P_5\|_{L^1},~\|P_{5x}\|_{L^1}\leq C\|m\|_{H^1},\label{u0}
\end{align}
and 
\begin{align}
	&\|P_2\|_{L^2}\leq C\|e^{-|\cdot|}\|_{L^2}\|m^2\|_{L^1}\leq C\|m\|_{H^1}&,
	\\
	~~~&\|P_{2x}\|_{L^2}\leq C\| sgn(\cdot)e^{-|\cdot|}\|_{L^2}\|m^2\|_{L^1}\leq C\|m\|_{H^1}&\label{u1}
	\\
	&\|P_3\|_{L^2}\leq C\|m\|_{H^1}&.\label{u2}
\end{align}
From \eqref{u1} and \eqref{u2}, one obtains
\begin{align}
	\|u\|_{L^2}\leq C\|m\|_{H^1}.\label{u}
\end{align}
Similarly, 
\begin{align}
    \|P_4\|_{L^1},\|P_{4x}\|_{L^1}\leq \|H\|_{L^1}\leq C(\|m\|_{H^1}+\|m\|_{H^1}^2).\label{u3}
\end{align}
By \eqref{u0}-\eqref{u3}, we can deduce that there exists a constant C depending on $\lambda$ such that
\begin{align}
	\|F\|_{L^1}\leq C(\|m\|_{H^1}+\|m\|_{H^1}^2)\triangleq C_S.
\end{align}
Assuming that $t>\tau$ and y=x($\tau$,$\beta$), we obtain
\begin{align}
	\mu_{(t)}(]-\infty,y-C_\infty(t-\tau)[)&=
		\mu_{(\tau)}((-\infty,y))-\mu_{(\tau)}((-\infty,y))+\mu_{(\tau)}((-\infty,x(t,\beta)))
		\notag\\
	&~~~+\mu_{(t)}((-\infty,y))-\mu_{(t)}((-\infty,y))-\mu_{(t)}((y-C_\infty(t-\tau),y))
	\notag\\
	&\leq \mu_{(\tau)}((-\infty,y))+\mu_{(\tau)}((-\infty,x(t,\beta)))-\mu_{(\tau)}(]-\infty,y[)
	\notag\\
		&~~~+\mu_{(t)}((x(t,\beta),y))-\mu_{(t)}\mu_{(t)}((y-C_\infty(t-\tau),y))
		\notag\\
		&\leq \mu_{(\tau)}((-\infty,y))+\int_{\tau}^{t}\|F\|_{L^1}dt^\prime\leq \mu_{(\tau)}((-\infty,y))+C_S(t-\tau).
\end{align}
Let $y^-(t)\triangleq y-(C_\infty+C_S)(t-\tau)$, we get
\begin{align}
	y^-+\mu_{(t)}(]-\infty,y^-(t)[)&\leq y-(C_\infty-C_S)(t-\tau)+\mu_{(\tau)}+C_S(t-\tau)
	\notag\\
	&\leq y-\mu_{(t)}(]-\infty,y[)\leq \beta
\end{align}
Which implies $$x(t,\beta)\geq y^-(t)=y-(C_\infty+C_S)(t-\tau).$$
Likewise, if we define  $y^+(t)\triangleq y+(C_\infty+C_S)(t-\tau)$, we can deduce that $$x(t,\beta)\leq y^+(t)=y+(C_\infty+C_S)(t-\tau).$$
\end{proof}

\begin{lemm}\label{lemm2}
	Let $m(t,x)\in H^1(\mathbb{R})$ be a conservative solution of the Cauchy problem \eqref{ceg1}. Then, for
	any $\bar{y}\in \mathbb{R}$, there exists a unique Lipschitz continuous map $t\mapsto x(t)$ which satisfies
	both \eqref{ceg4} and \eqref{F1}. Moreover, for any $0\leq t\leq \tau$ one has
	\begin{align}\label{3.2}
		m(t,x(t))-m(\tau,x(\tau))=-\int_{\tau}^{t}F(s,x(s))ds.
	\end{align}
\end{lemm}
\begin{proof}
	\begin{itemize}
		\item [1.]	Using the adapted coordinates $(t,\beta)$ as in \eqref{ceg4}, we asumme that x(t) is the  charateristic starting at $\bar{y}$ in the form of $t\mapsto x(t)=x(t,\beta(t))$, where $\beta(t)$ is a map to be determined. We sum \eqref{ceg4} and \eqref{F1}, and integrating w.r.t. time we obtain
		\begin{align}\label{21}
			x(t)+\int_{-\infty}^{x(t)}m_x^2(t,x)dx=\bar{y}+\int_{-\infty}^{\bar{y}}m_{0,x}^2dx+\int_{0}^{t}(u(t,x(t))+\int_{-\infty}^{x(t)}Edx)dt.
		\end{align}
		Introducing the function 
		\begin{align}\label{G}
			G(t,\beta)\triangleq \int_{-\infty}^{x(t,\beta)} u_x+Edx
		\end{align}
		and the constant
		\begin{align}
			\bar{\beta}=\bar{y}+\int_{-\infty}^{\bar{y}}m_{0,x}^2dx,
		\end{align}
		we can rewrite the equation \eqref{21} as the form of 
		\begin{align}
			\beta(t)=\bar{\beta}+\int_{0}^{t}G(s,\beta(s))ds\label{beta}.
		\end{align} 
	\item[2.] Next, we cliam that For each fixed $t\geq 0$, the function $\beta \mapsto G(t,\beta)$ defined at \eqref{G} is uniformly bounded and absolutely continuous. Moreover,
	\begin{align}
		G_\beta=[u_x+E]x_\beta=\frac{u_x+E}{1+m_x^2}\in [-C,C]
	\end{align}
	for some constant C depending only on the norm of $m$ and $\lambda$.
	
	To prove the claim, we notice that 
	\begin{align}
		\|G_\beta\|_{L^\infty}=\|\frac{u_x+E}{1+m_x^2}\|_{L^\infty}\leq\|\frac{u_x}{1+m_x^2}\|_{L^\infty}+\|\frac{E}{1+m_x^2}\|_{L^\infty}.
	\end{align}
	By \eqref{p1}-\eqref{p3}, we can deduce that 
	\begin{align}
		\|\frac{u_x}{1+m_x^2}\|_{L^\infty} \leq C.
	\end{align}
	Using the fact that 
	\begin{align}
		\frac{m_x}{1+m_x^2},\frac{m_x^2}{1+m_x^2}\leq 1,
	\end{align}
	it is not hard to check that for any $i=1,2,3,4,5$, the following inequalities holds
	\begin{align}
		\|\frac{m_xP_i}{1+m_x^2}\|_{L^\infty},\|\frac{m_xP_{ix}}{1+m_x^2}\|_{L^\infty} \leq C,\label{P1}
	\end{align}
	and by \eqref{p1}-\eqref{p3} and \eqref{mx},we can deduce that 
	\begin{align}
		\|\frac{mu_xm_x}{1+m_x^2}\|_{L^\infty}&\leq C\|m\|_{L^\infty}\|\frac{m_x^2-mm_x+\frac{1}{2\lambda}(P_{2x}+P_3-m^2-P_2)m_x}{1+m_x^2}\|_{L^\infty}
		\notag\\
		&\leq C\|m\|_{L^\infty} (\|\frac{m_x^2}{1+m_x^2}\|_{L^\infty}+\|m\|_{L^\infty}\|\frac{m_x}{1+m_x^2} \|_{L^\infty}
		\notag\\
		&~~~+\frac{1}{2\lambda}\|P_{2x}+P_3-m^2-P_2\|_{L^\infty}\|\frac{m_x}{1+m_x^2} \|_{L^\infty}
		\notag\\
		&\leq C.\label{M1}
	\end{align}
	Moreover, one can derive that
	\begin{align}
		\|\frac{m_x(P_{5x}+P_{5xx})}{1+m_x^2}\|_{L^\infty}&=\|\frac{m_x(\partial_x^2+\partial_x)G^{-1}G^{-1}m_x^2}{1+m_x^2}\|_{L^\infty}
		\notag\\
		&\leq C\|G^{-1}m_x^2+G^{-1}G^{-1}m_x^2+\partial_xG^{-1}G^{-1}m_x^2\|_{L^\infty}
		\notag\\
		&\leq C\|m_x^2\|_{L^1}\leq C,
		\end{align}
	and
	\begin{align}
		\|\frac{Hm_x}{1+m_x^2}\|_{L^\infty}&\leq	\|\frac{um_x^2}{1+m_x^2}\|_{L^\infty}	+\|\frac{umm_x}{1+m_x^2}\|_{L^\infty}
		\notag \\
		&\leq C(\|u\|_{L^\infty}+\|u\|_{L^\infty}\|m\|_{L^\infty}).\label{H1}
	\end{align}
	Combining \eqref{P1}-\eqref{H1}, we conclude that
	\begin{align}
		\|\frac{E}{1+m_x^2}\|_{L^\infty}\leq C,
	\end{align}
	which completed the proof of our claim.
	
	Hence the function $G$ in \eqref{G} is uniformly Lipschitz continuous w.r.t. $\beta$.
	\item[3.] Thanks to the Lipschitz continuity of the function G, the existence of a unique solution to the integral equation \eqref{beta} can be proved by a standard fixed point argument. Namely, consider the Banach space of all continuous functions $\beta: \mathbb{R}_+\mapsto \mathbb{R}$ with weighted norm $$\|\beta\|_\star\triangleq \sup_{t\geq 0}e^{-2Ct}|\beta(t)|.$$
	On this space, we claim that the Picard map 
	\begin{align}
		\mathscr{P}\beta(t)\triangleq \bar{\beta}+\int_{0}^{t}G(\tau,\beta(\tau))d\tau
	\end{align}
is a strict contraction. Indeed, assume $\|\beta-\tilde{\beta}\|=\delta>0$. This implies 
\begin{align}
	|\beta(\tau)-\beta(\tau)|\leq \delta e^{2C\tau}
\end{align}
Hence
\begin{align}
	|\mathscr{P}\beta(t)-\mathscr{P}\tilde{\beta}(t)|&=|\int_{0}^{t}(G(\tau,\beta(\tau))-G(\tau,\tilde{\beta(\tau)}))d\tau|
	\notag\\
	&\leq C\int_{0}^{t}|\beta(\tau)-\tilde{\beta(\tau)}|d\tau\leq\int_{0}^{t}C\delta e^{2C\tau}d\tau\leq \frac{\delta}{2}e^{2Ct}.
\end{align}
Then we conclude $\|\mathscr{P}\beta-\mathscr{P}\tilde{\beta}\|_\star\leq \frac{\delta}{2}$.The contraction mapping principle guarantees that \eqref{beta} has
a unique solution. Thanks to the arbitrary of the $T$, we infer that the integral equation \eqref{beta} has a
unique solution on $\mathbb{R}_+$.
\item[4] By the previous construction, the map $t\mapsto x(t)\triangleq x(t,\beta(t))$ defined by \eqref{def1} provides the unique solution to \eqref{beta}. Due to the Lipschitz continuity of $\beta(t)$ and x(t)=x(t,$\beta(t)$), $\beta(t)$ and $x(t)$ are differentiable almost everywhere, so we only have to consider the time where $x(t)$ is differentiable.
 It suffices to show that \eqref{F1} holds at almost every time. Assume, on the contrary, $\dot{x}(\tau)\neq u(\tau,x(\tau))$. Without loss of generality, let 
 \begin{align}
 	\dot{x}(\tau)=u(\tau,x(\tau))+2\epsilon_0
 \end{align}
for some $\epsilon_0>0$. The case $\epsilon_0<0$ is entirely similar. To derive a contradiction we observe that for all $t\in (\tau,\tau+\delta)$, with $\delta>0$ small enough one has 
\begin{align}
	x^+(t)\triangleq x(\tau)+(t-\tau)[u(\tau,x(\tau))+\epsilon_0]<x(t).
\end{align} 
We also observe that if $\phi$ is Lipschitz continuous with compact support, then the approximation argument guarantees that
\eqref{bal}  remain hold for any test function $\phi\in H^1(\mathbb{R})$ with compact support.	
For any $\epsilon>0$  enough small, we give the following functions
\begin{equation}\label{eqa17}
	\varrho^{\varepsilon}(s,x)=\left\{
	\begin{array}{rcl}
		0, &y\leq -{\varepsilon}^{-1},\\
		x+{\varepsilon}^{-1}, & ~~-{\varepsilon}^{-1}\leq y \leq 1-{\varepsilon}^{-1},\\
		1, & ~~1-{\varepsilon}^{-1}\leq y \leq y^+(x),\\
		1-{\varepsilon}^{-1}(x-x(s)),  &~~~~~~ y^+(s)\leq y\leq y^+(s)+\varepsilon,\\
		0, & ~~~~~~y\geq y^+(s)+\varepsilon,
	\end{array} \right.
\end{equation}
\begin{equation}\label{eqa18}
	\chi^{\varepsilon}(s)=\left\{
	\begin{array}{rcl}
		0, &~~~~~~~~~~s\leq \tau-{\varepsilon}^{-1},\\
		{\varepsilon}^{-1}(s-\tau+\varepsilon), & ~~\tau-{\varepsilon}\leq s \leq \tau,\\
			1, & ~~\tau \leq s \leq t,\\
		1-{\varepsilon}^{-1}(s-t),  &~~ t\leq s\leq t+\varepsilon,\\
		0, & ~~~~~~~s\geq t+\varepsilon.
	\end{array} \right.
\end{equation}
Define
\begin{align*}
	\phi^{\varepsilon}(s,x)=\min\{ {\varrho^{\varepsilon}(s,x),	\chi^{\varepsilon}(s)}\}.
\end{align*}
Using $\phi^\epsilon$ as test function in \eqref{bal} we obtain
\begin{align}\label{bb}
	\int_{0}^{\infty}\int [m_x^2\phi^\epsilon_t+um_x^2\phi^\epsilon_x+2E\phi^\epsilon]dxdt+\int m_{0,x}^2\phi^\epsilon(0,x)dx =0.
\end{align}
	If $t$ is sufficiently close to $\tau,$ we have
\begin{align*}
	\lim\limits_{\varepsilon\rightarrow 0}\int_{\tau}^{t}\int_{x^+(s)-\varepsilon}^{x^+(s)+\varepsilon}m_x^2\phi_t^\epsilon +(um_x^2)\phi_x^\epsilon +2E\phi^\epsilon dyds\geq 0.
\end{align*}
	Using the fact that $m(s,x) <m(\tau)+{\epsilon}_0$   and ${\phi}_x^{\epsilon}\leq 0.$ For any $s\in[\tau+\varepsilon,t-\varepsilon],$ we infer that
\begin{align}\label{5.139}
	0={\phi}_t^{\varepsilon}+[u(\tau,x(\tau))+\epsilon_0]{\phi}_x^{\varepsilon}\leq {\phi}_t^{\varepsilon}+(u(s,x(s))){\phi}_x^{\varepsilon}.
\end{align}
Noticing that the family of measures
${\mu}_t$ depends continuously on $t$ in the topology of weak convergence, taking the limit of \eqref{bb} as $\epsilon\rightarrow 0$, we obtain
\begin{align}\label{37}
	0=&\int_{-\infty}^{x(\tau)}m^2_x(\tau,y)dy-\int_{-\infty}^{x^+(t)}m^2_x(t,x)dx+\int_{\tau}^{t}\int_{-\infty}^{x^+(s)}2E\phi^{\varepsilon}(s,x)dxds\notag\\
	&+\lim\limits_{\varepsilon\rightarrow 0}	\int_{\tau}^t\int_{{x^+}(s)-\varepsilon}^{{x^+}(s)+\varepsilon}m_x^2(\phi_t^\epsilon+u\phi_x^\epsilon)(s,x)dxds\notag\\
	\geq&\int_{-\infty}^{x(\tau)}m^2_x(\tau,x)dx-\int_{-\infty}^{x^+(t)}m^2_x(t,x)dx+\int_{\tau}^{t}\int_{-\infty}^{x^+(s)}2E(s,x)dxds.
\end{align}
In turn,\eqref{37} implies 
\begin{align}
	\int_{-\infty}^{x^+(t)}m^2_x(t,x)dx&\geq \int_{-\infty}^{x(\tau)}m^2_x(\tau,x)dx+\int_{\tau}^{t}\int_{-\infty}^{x^+(s)}2E(s,x)dxds.\\
	&=\int_{-\infty}^{x(\tau)}m^2_x(\tau,x)dx+\int_{\tau}^{t}\int_{-\infty}^{x(s)}2E(s,x)dxds+o_1(t-\tau).
\end{align}
Notice that the last term is a higher order infinitesimal, satisfying $\frac{o_1(t-\tau)}{t-\tau} \rightarrow 0$ as $t\rightarrow \tau$. Indeed
\begin{align}
	|o_1(t-\tau)|=\int_{\tau}^{t}\int_{-\infty}^{x(s)}2E(s,x)dxds.
\end{align} 
For simplicity, we only show the estimate of the term $(mu)_xm_x$ in E. Firstly, we have
\begin{align}
	\int_{\tau}^{t}\int_{x^+(s)}^{x(s)}(um_x^2)dxds&\leq \int_{\tau}^{t}\|u\|_{L^\infty}\int_{x(s)}^{x^+(s)}m_x^2dxds\leq\|u\|_{L^\infty}\int_{x^+(s)}^{x(s)}m_x^2dxds
	\notag\\
	& \leq \|u\|_{H^1}\int_{\tau}^{t}\|m_x^2\|_{L^1}(x(s)-x^+(s))ds
	\leq C(t-\tau)^2.
\end{align}
Similarly
\begin{align}
	\int_{\tau}^{t}\int_{x(s)}^{x^+(s)}(-mu_xm_x)dxds&= \int_{\tau}^{t}\int_{x(s)}^{x^+(s)}-m[m_x-m+\frac{1}{2\lambda}(G^{-1}m_x^2+G^{-1}m^2-\partial G^{-1}m^2)]m_xdxds
	\notag\\
	&\leq \|m\|_{H^1}\int_{\tau}^{t}\|m_x^2\|_{L^1}(x(s)-x^+(s))ds
	\notag\\
	&+(\|m\|_{H^1}^2+\|m\|_{H^1}^3)\int_{\tau}^{t}\|m_x\|_{L^2}(x(s)-x^+(s))^{\frac{1}{2}}ds
	\notag\\
	&\leq C(t-\tau)^{\frac{3}{2}}+C(t-\tau)^2.
\end{align}
Thus we obtain that
\begin{align}
	|\circ_1(t-\tau)|\leq C((t-\tau)^\frac{3}{2}+(t-\tau)^2).
\end{align}
For $t$ sufficiently close to $\tau$, we have
\begin{align}
	\beta(t)=&\beta(\tau)+(t-\tau)[u(\tau,x(\tau))+\int_{-\infty}^{x(\tau)}Fdx]+o_2(t-\tau)
	\notag\\
	=& x(t)+\int_{-\infty}^{x(t)}m_x^2dx=x(t)+\mu_{(t)}(-\infty,x(t))
	\notag\\
	>&x(\tau)+(t-\tau)[u(\tau)+\epsilon_0]+\mu_{(t)}(-\infty,x^+(t))
	\notag\\
	\geq&x(\tau)+(t-\tau)[u(\tau)+\epsilon_0]+\mu_{(t)}(-\infty,x(t))
	+\int_{\tau}^{t}\int_{-\infty}^{x(s)}Fdxds+o_1(t-\tau).
\end{align} 
We can deduce that 
\begin{align}
	\beta(\tau)&+(t-\tau)[u(\tau,x(\tau))+\int_{-\infty}^{x(\tau)}Fdx]+o_2(t-\tau)
	\notag\\
	\geq &[x(\tau)+\int_{-\infty}^{x(\tau)}m_x^2dx]+(t-\tau)[u(\tau,x(\tau))+\epsilon_0]+\int_{\tau}^{t}\int_{-\infty}^{x(s)}Fdxds+o_1(t-\tau).
\end{align}
Subtracting common terms, dividing both sides by $t-\tau$ and letting $t\rightarrow\tau$, we achieve a contradiction. 
Therefore, \eqref{F1} must hold.
\item[5] We now prove \eqref{3.2}. For every test function $\phi \in \mathbb{C}^\infty_c(\mathbb{R}^2)$, one has
\begin{align}
	\int_{0}^{\infty}\int m\phi_t-um_x\phi+F\phi dxdt+\int m_0\phi(0,x)dx=0.
\end{align}
Given any $\psi \in \mathbb{C}_c^\infty$, let $\phi=\psi_x$. Since the map $x \mapsto m(t,x)$ is absolutely continuous, we can intergrate by parts w.r.t. x and obtain
\begin{align}\label{psi1}
	\int_{0}^{\infty}\int m_x\psi+um_x\psi_x+F\psi_xdxdt+\int \bar{m}\psi dx=0.
\end{align}
For any  $\epsilon\geq0$  sufficiently small, we give the following function
\begin{equation*}
	\varrho^{\varepsilon}(s,x)=\left\{
	\begin{array}{rcl}
		0, &~~~~y\leq -{\varepsilon}^{-1},\\
		x+{\varepsilon}^{-1}, &~~~~~~ -{\varepsilon}^{-1}\leq x \leq 1-{\varepsilon}^{-1},\\
		1, &1-{\varepsilon}^{-1}\leq x\leq x(s),\\
		1-{\varepsilon}^{-1}(x-x(s)),  & ~~~~~~~~~x(s)\leq x\leq x(s)+\varepsilon,\\
		0 & ~~~~~~~~x\geq x(s)+\varepsilon,
	\end{array} \right.
\end{equation*}
and
\begin{align*}
	\psi^{\varepsilon}(s,x)=\min\{ {\varrho^{\varepsilon}(s,x),	\chi^{\varepsilon}(s)}\},
\end{align*}
We now use the test function $\phi=\psi^\epsilon$ in \eqref{psi1} and let $\epsilon \mapsto 0$. 
\begin{align}
	\int_{-\infty}^{x(t)}m_x(t,x)dx=\int_{-\infty}^{x(\tau)}m_x(\tau,x)dx-\int_{\tau}^{t}F(s,x(s))ds+\lim_{\epsilon\rightarrow0}\int_{\tau-\epsilon}^{t+\epsilon}\int_{x(s)}^{x(s)+\epsilon}m_x(\psi_t^\epsilon+u\psi_x^\epsilon)dxds
\end{align}
	Hence, it is shown that
\begin{align}
	&\lim\limits_{\varepsilon\rightarrow 0}	\int_{\tau-\varepsilon}^{t+\varepsilon}\int_{{x}(s)}^{{x}(s)+\varepsilon}m_x(\psi_t^\epsilon+u\psi_x^\epsilon)dyds\notag\\
	&=\lim\limits_{\varepsilon\rightarrow 0}	\Big(\int_{\tau-\varepsilon}^{\tau}+\int_{\tau}^{t}+\int_{t}^{t+\varepsilon}\Big)\int_{{x}(s)}^{{x}(s)+\varepsilon}m_x(\psi_t^\epsilon+u\psi_x^\epsilon)dyds=0.
\end{align}
Taking advantage of  Cauchy's inequality and $m_x\in L^2$, one has
\begin{align*}
	&|\int_{\tau}^{t}\int_{{x}(s)}^{{x}(s)+\varepsilon}m_x[{\psi}_t^{\varepsilon}+u{\psi}_x^{\varepsilon}]dxds|\notag\\&\leq \int_{\tau}^{t}\Big(\int_{{x}(s)}^{{x}(s)+\varepsilon}|m_x|^2dy\Big)^{\frac{1}{2}}	\Big(\int_{{x}(s)}^{{x}(s)+\varepsilon}[{\psi}_t^{\varepsilon}+u{\psi}_x^{\varepsilon}]^2dx\Big)^{\frac{1}{2}}ds.
\end{align*}
For each $\epsilon>0$ consider the function 
\begin{align}
	\eta_\epsilon(s)\triangleq (\sup_{x\in \mathbb{R}}\int_{x}^{x+\epsilon}m_x^2(s,y)dy)^\frac{1}{2}.
\end{align}
Observe that all functions $\eta_\epsilon$ are uniformly bounded. By the dominated convergence theorem,
\begin{align}\label{con1}
\lim_{\epsilon\rightarrow0}\int_{\tau}^{t}(\int_{x}^{x+\epsilon}m_x^2(s,y)dy)^\frac{1}{2}ds\leq \lim_{\epsilon\rightarrow0}\int_{\tau}^{t}\eta_\epsilon(s)ds=0
\end{align}
For $s\in[\tau,t]$, we construct the function 
\begin{align}
	\psi^\epsilon_x(s,y)=-\epsilon^{-1},
\end{align}
and 
\begin{align}
	\psi_t^\epsilon(s,y)+u(s,x(s))\psi_x^\epsilon(s,y)=0~~~~,~~~~x(s)<y<x(s)+\epsilon.
\end{align}
This implies
\begin{align}\label{con2}
	&\int_{x(s)+\epsilon}^{x(s)}|\psi_t^\epsilon(s,y)+u(s,y)\psi^\epsilon_x(s,y)|^2dy=\epsilon^{-2}\int_{x(s)}^{x(s)+\epsilon}|u(s,y)-u(s,x(s))|^2dy 
	\notag\\
	&\leq\epsilon^{-1}(\max_{x(s)\leq y\leq x(s)+\epsilon}|u(s,y)-u(s,x(s))|)^2\leq\epsilon^{-1}(\int_{x(s)}^{x(s)+\epsilon}|u_x(s,y)|dy)^2
	\leq \epsilon^{-1}(\epsilon^{\frac{1}{2}}\|u_x(s)\|_{L^2})^2
	\notag\\
	&\leq\|m\|_{H^1}
\end{align}
Combining \eqref{con1} and \eqref{con2}, we show the estimate the integral
\begin{align}
	&(\int_{\tau-\epsilon}^{\tau}+\int_{t}^{t+\epsilon})\int_{x(s)}^{x(s)+\epsilon}m_x(\psi_t+u\psi)dxds
	\notag\\
	&\leq(\int_{\tau-\epsilon}^{\tau}+\int_{t}^{t+\epsilon})(\int_{x(s)}^{x(s)+\epsilon}m_x^2dy)^{\frac{1}{2}}(\int_{x(s)}^{x(s)+\epsilon}(\psi_t+u\psi)^2dy)^{\frac{1}{2}}
	\notag\\
	& \leq 2\epsilon\|m\|_{H^1}(\int_{x(s)}^{x(s)+\epsilon}\epsilon^{-2}Cdy)^{\frac{1}{2}}\leq C\epsilon^{\frac{1}{2}}\rightarrow 0
\end{align}
as $\epsilon \rightarrow 0$.The above analysis has shown that
\begin{align}
	\lim_{\epsilon\rightarrow 0}\int_{\tau-\epsilon}^{t+\epsilon}m_x(\psi_t+u\psi)ds=0.
\end{align}
Therefore we arrive at \eqref{3.2}.
\item[6.] Using the uniqueness of $\beta$, we can prove uniqueness of $x(t)$. 
	\end{itemize}
\end{proof}
 \begin{lemm}\label{lemm3}
 	Let m=m(t,x) be a conservative solution of \eqref{ceg1}. Then the map$(t,\beta)\mapsto m(t,\beta)\triangleq m(t,x(t,\beta))$ is Lipschitz continuous, with a constant depending only on the norm $\|m_0\|_{H^1}$
 \end{lemm}
\begin{proof}
	Combining \eqref{MMM} and \eqref{3.2}, we obtain
	\begin{align}
		|m(t,x(t,\bar{\beta}))-m(\tau,\bar{\beta})|\leq|m(t,x(t,\bar{\beta}))-m(t,\beta(t))|+|m(t,x(t,\beta(t)))-m(\tau,\beta(\tau))|
		\notag\\
		\leq \frac{1}{2}|\beta(t)-\bar{\beta}|+(t-\tau)\|F\|_{L^\infty}\leq(t-\tau)(\|G\|_{L^\infty}+\|F\|_{L^\infty}).
	\end{align}
\end{proof}

\section{Proof of Theorem\ref{main}}
	We need to seek  a good characteristic, and employ how the gradient $m_x$ of a conservative solution varies along the good characteristic, and complete the  proof of uniqueness.
	\begin{proof}\textbf{Step 1.} Lemmas \ref{lemm1}-\ref{lemm3} ensure that the map $(t,\beta)\mapsto (x,u)(t,\beta)$ and $\beta\mapsto F(t,\beta)$ are Lipschitz continuous. Thanks to  Rademacher's theorem, the partial derivatives $x_t, x_{\beta}, m_t,m_{\beta}$ and $F_{\beta}$ exist almost everywhere. Moreover,  $x(t,\beta)$
		is the unique solution to \eqref{F1}, and the following holds.

		$\textbf{(GC)}$ For   $a.e.~ t\geq0$, the point $(t,\beta(t,\bar{\beta}))$ is a Lebesgue point for the partial derivatives $x_t,x_{\beta},m_t,m_{\beta}$
	and $F_{\xi}$. Moreover, $x_{\beta}(t,\beta)>0$ for $a.e. ~t\geq0.$
	
	If 	$\textbf{(GC)}$ holds, then  $t \rightarrow x(t,\bar{\beta})$ is a good characteristic.
	
	\textbf{Step 2.}  We now construct an O.D.E. to describe that the quantities $m_{\beta}$ and $x_{\beta}$ vary along a good characteristic. Supposing that $t,~\tau\notin\mathcal{N},$ and $y(t,\xi)$ is a good characteristic, we then have
	\begin{align*}
		x(t,\beta(t,\bar{\beta}))=\bar{x}(t,\bar{\beta})+\int_\tau^tu(s,x(s,\beta(t;\tau,\bar{\beta}))) ds.
	\end{align*}
    Where $\mathcal{N}$ denotes a null set with meas($\mathcal{N}$)=0 such that for every $t\notin \mathcal{N}$ the measure $\mu_{(t)}$ is absolutely  continous and has density $m_x^2(t,\dot )$.
    
	Differentiating the above equation with respect to $\bar{\xi}$, we deduce that
	\begin{align}\label{p35}
		x_{\beta}\frac{\partial}{\partial\bar{\beta}}\beta(t;\tau,\bar{\beta})=x_{\beta}(\tau,\bar{\beta})+\int_\tau^tu_\beta (s,\beta(t;\tau,\bar{\beta}))\frac{\partial}{\partial\bar{\beta}}\beta(t;\tau,\bar{\beta})ds.
	\end{align}
	Likewise, we have
	\begin{align}\label{p36}
	m_{\beta}\frac{\partial}{\partial\bar{\beta}}\beta(t;\tau,\bar{\beta})=m_{\beta}(\tau,\bar{\beta})+\int_\tau^tF_\beta (s,\beta(t;\tau,\bar{\beta}))\frac{\partial}{\partial\bar{\beta}}\beta(t;\tau,\bar{\beta})ds.
	\end{align}
	From \eqref{p35}-\eqref{p36}, we end up with
	\begin{equation}\label{qp37}
		\left\{\begin{aligned}
			&\frac{d}{dt}x_\beta+G_\beta x_\beta=u_\beta,\\
			&\frac{d}{dt}m_\beta+G_\beta m_\beta=F_\beta.
		\end{aligned}\right.
	\end{equation}

	\textbf{Step 3.} We now  return to the original coordinates $(t, x)$ and derive an evolution equation for the
	partial derivative $k_x$ along a “good” characteristic curve.
	For a fixed point $(t, x)$ with $t\notin \mathcal{N}$. Suppose
	that $\bar{x}$ is a Lebesgue point for the map $x\rightarrow m_x(t, x)$, and $\bar{\beta}$ satisfies $\bar{x} = x(t, \bar{\beta}),$ and suppose
	that $t\rightarrow \beta(t;\tau,\bar{\beta})$ is a good characteristic, which implies $\textbf{(GC)}$ holds. We observe that
	\begin{align}\label{p38}
		m_{\beta}(t,x)=\frac{1}{x^2_\beta(t,\bar{\beta})}-1>0,\ x_\beta(t,\bar{\beta})>0,
	\end{align}
	which implies that $y_{\xi}(t,\xi)>0.$
	
	Hence,  the partial derivative
	$m_x$ and $u_x$ can be calculated as shown below
	\begin{align*}
		m_x(t,x(t,\beta(t;\tau,\bar{\beta})))=\frac{m_{\beta}(t,\beta(t;\tau,\bar{\beta}))}{x_{\beta}(t,\beta(t;\tau,\bar{\beta}))}.~~~~~	u_x(t,x(t,\beta(t;\tau,\bar{\beta})))=\frac{u_{\beta}(t,\beta(t;\tau,\bar{\beta}))}{x_{\beta}(t,\beta(t;\tau,\bar{\beta}))}.
	\end{align*}
	Applying \eqref{qp37} to describe the evolution of $m_{\beta}$ and $x_{\beta},$ we  infer that the map $t\rightarrow m_x(t,x(t,\beta(t,{\bar{\beta}})))$ is absolutely continuous. It follows that
	\begin{align*}
		\frac{d}{dt}m_x(t,x(t,\beta(t,{\bar{\beta}})))&=\frac{d(\frac{m_{\beta}}{x_{\beta}})}{dt}=\frac{x_\beta \frac{d}{dt}m_\beta-m_\beta\frac{d}{dt}x_\beta}{{x_{\beta}}^2}=\frac{x_\beta (F_\beta-m_\beta G_\beta)-m_\beta(u_\beta-G_\beta x_\beta)}{{x_{\beta}}^2}
		\notag\\
		&=\frac{F_\beta x_\beta-u_\beta m_\beta}{x_\beta^2}=\frac{F_\beta}{x_\beta}-u_xm_x.
	\end{align*}
	Hence, we conclude that as long as $y_{\beta}\neq 0$, the map $t\rightarrow k_x$  is absolutely continuous.
	
	\textbf{Step 4.} Consider the function
	\begin{equation}
		v\triangleq
	\left\{\begin{aligned}
		&2\arctan m_x~~~~~~if~~~0<x_\beta\leq1\\
		&~~~~~~~~~\pi ~~~~~~~~~if~~~x_\beta=0.\\
	\end{aligned}\right.
	\end{equation}
    Observe that this implies
    \begin{align}
    	x_\beta=\frac{1}{1+m_x^2}=\cos^2\frac{v}{2},~~~\frac{m_x}{1+m_x^2}=\frac{1}{2}\sin v,~~~\frac{m_x^2}{1+m_x^2}=\sin^2\frac{v}{2}
    \end{align}
	From which it follows that	
	\begin{equation}
		\left\{\begin{aligned}
			&\frac{d}{dt}\beta(t,\bar{\beta})=G,\\
			&\frac{d}{dt}x(t,\beta(t,\bar{\beta}))=u(t,\beta(t,\bar{\beta})),\\
			&\frac{d}{dt}m(t,\beta(t,\bar{\beta}))=F(t,\beta(t,\bar{\beta})),\\
			&\frac{d}{dt}v(t,\beta(t,\bar{\beta}))=2\cos^{2} \frac v 2 P-N\sin v-\sin^{2} \frac v 2.\\
		\end{aligned} \right.\label{Ky1}
	\end{equation}	
with
$N(t,\beta)=\big(-m-P_1-P_{1x}+\frac 1 {2\lambda} (-P_2+P_{3x}+P_{2x}-m^2)\big)(t,\beta)$ and~$P=F+mN+\lambda u+\frac {m^2} {2}$
The function $P_i$ admits a representation of the variavle $\beta$, namely
\begin{align}\label{3P}
	\left\{\begin{array}{lllll}
		P_1=-\frac 1 2 \int_{-\infty}^{+\infty} e^{-|\int_{\beta'}^{\beta} \cos^{2} \frac v 2  ds|} m\cos^{2} \frac v 2 d\beta',  \\
		P_{1x}=-\frac 1 2 \big(\int_{\beta}^{+\infty}-\int_{-\infty}^{\beta}\big) e^{-|\int_{\beta'}^{\beta} \cos^{2} \frac v 2  ds|} m\cos^{2} \frac v 2 d\beta',  \\
		P_2=-\frac 1 2 \int_{-\infty}^{\infty} e^{-|\int_{\beta'}^{\beta} \cos^{2} \frac v 2  ds|} m^2\cos^{2} \frac v 2 d\beta',  \\
		P_{2x}=-\frac 1 2 \big(\int_{\beta}^{+\infty}-\int_{-\infty}^{\beta}\big) e^{-|\int_{\beta'}^{\beta} \cos^{2} \frac v 2 ds|} m^2\cos^{2} \frac v 2 d\beta',  \\
		P_3=-\frac 1 2 \int_{-\infty}^{+\infty} e^{-|\int_{\beta'}^{\beta} \cos^{2} \frac v 2 ds|} \sin^{2} \frac v 2d\beta',  \\
		P_{3x}=-\frac 1 2 \big(\int_{\beta}^{+\infty}-\int_{-\infty}^{\beta}\big) e^{-|\int_{\beta'}^{\beta} \cos^{2} \frac v 2 ds|} \sin^{2} \frac v 2 d\beta',  \\
		P_4=-\frac 1 2 \int_{-\infty}^{+\infty} e^{-|\int_{\beta'}^{\beta} q\cos^{2} \frac v 2  ds|}\frac  2 u\sin v-mu\cos^2 \frac v 2d\beta', \\
		P_{4x}=-\frac 1 2 \big(\int_{\beta}^{+\infty}-\int_{-\infty}^{\beta}\big) e^{-|\int_{\beta'}^{\beta} \cos^{2} \frac v 2  ds|} \frac  2 u\sin v-mu\cos^2 \frac v 2d\beta',  \\
		P_5=-\frac 1 2 \int_{-\infty}^{+\infty} e^{-|\int_{\beta'}^{\beta} \cos^{2} \frac v 2 ds|} P_3\cos^{2} \frac v 2 d\beta',  \\
		P_{5x}=-\frac 1 2 \big(\int_{\beta}^{+\infty}-\int_{-\infty}^{\beta}\big) e^{-|\int_{\beta'}^{\beta} \cos^{2} \frac v 2 ds|} P_3\cos^{2} \frac v 2 d\beta'.  \\
	\end{array}\right.
\end{align}
	For any $\bar{\beta}\in \mathbb{R},$ we deduce that the following initial  conditions
	\begin{equation}
		\left\{\begin{aligned}
			&\beta(0,\bar{\beta})=\bar{\beta},\\
		&\beta(0,\bar{\beta})=x(0,\bar{\beta}),\\
		&m(0,\bar{\beta})=m_0(x(0,\bar{\beta})),\\
		&v(0,\bar{\beta})=2\arctan u_{0,x}(x(0,\bar{\beta})).\\
		\end{aligned}\right.\label{K01}
	\end{equation}
	Making use of  all coefficients is Lipschitz continuous and the previous steps again, the system  \eqref{Ky1}-\eqref{K01} has a unique globally solution.
	
	\textbf{Step 5.} Let $m$ and $\tilde{m}$ be two conservative weak solution of  \eqref{ceg1} with the same
	initial data $\bar{m}\in H^1(\mathbb{R})$
	. For $a.e. ~t\geq0,$  the corresponding Lipschitz continuous maps $\beta \mapsto x(t,\beta), {\beta} \mapsto \tilde{x}(t,\beta)$ are
	strictly increasing. Hence they have continuous inverses, say $x\mapsto{\beta}^{-1}(t,x),
	x\mapsto{\tilde{\beta}}^{-1}(t,x)$.
	Thus, we deduce that
	$$x(t,\beta)=\tilde{x}(t,\beta),\  u(t,x(t,\beta))=\tilde{u}(t,{x}(t,\beta)).$$
	Moreover, for $a.e.~ t\geq0$, we have
	$$m(t,x)=m(t,y(t,\beta))=\tilde{m}(t,\tilde{x}(t,\beta))=\tilde{m}(t,x).$$
	Then, we finish the proof of Theorem \ref{main}.
\end{proof}	
\smallskip
\noindent\textbf{Acknowledgments} This work was
partially supported by the National Natural Science Foundation of China (No.12171493).

\noindent\textbf{Data Availability.}
The data that support the findings of this study are available on citation. The data that support the findings of this study are also available
from the corresponding author upon reasonable request.

	\phantomsection
	\addcontentsline{toc}{section}{\refname}
	\bibliographystyle{abbrv} 
	\bibliography{Feneref}
\end{document}